\documentclass[11pt,letterpaper]{amsart}

\usepackage{graphicx}
\usepackage{amssymb}
\usepackage{amsthm}
\usepackage{amsmath}
\usepackage{stmaryrd}
\usepackage{cite}
\usepackage[mathscr]{eucal}
\usepackage{hyperref}
\usepackage{comment}
\usepackage{caption}
\usepackage{subcaption}
\usepackage{wrapfig}
\usepackage[permil]{overpic}
\usepackage{url}
\usepackage{hyperref}
\usepackage{placeins}

\title[An $SO(3)\times SO(8)$-invariant Einstein metric on $S^3\times S^7$]{An $SO(3)\times SO(8)$-invariant Einstein metric on $S^3\times S^7$}
\author[Huang]{Yuming Huang}
\address{Dpartment of Mathematics, University of California, Santa Barbara, South Hall, Room 6432P, Santa Barbara, CA 93106, USA}
\email{yuming\_huang@ucsb.edu}


\makeatletter
\@namedef{subjclassname@2026}{
\textup{2020} Mathematics Subject Classification}
\makeatother


\begin{document}
\newcommand{\Ext}{\bigwedge\nolimits}
\newcommand{\Div}{\operatorname{div}}
\newcommand{\Hol} {\operatorname{Hol}}
\newcommand{\diam} {\operatorname{diam}}
\newcommand{\Scal} {\operatorname{Scal}}
\newcommand{\scal} {\operatorname{scal}}
\newcommand{\Ric} {\operatorname{Ric}}
\newcommand{\Hess} {\operatorname{Hess}}
\newcommand{\grad} {\operatorname{grad}}
\newcommand{\Sect} {\operatorname{Sect}}
\newcommand{\Rm} {\operatorname{Rm}}
\newcommand{ \Rmzero } {\mathring{\Rm}}
\newcommand{\Rc} {\operatorname{Rc}}
\newcommand{\Curv} {S_{B}^{2}\left( \mathfrak{so}(n) \right) }
\newcommand{ \tr } {\operatorname{tr}}
\newcommand{ \id } {\operatorname{id}}
\newcommand{ \Riczero } {\mathring{\Ric}}
\newcommand{ \ad } {\operatorname{ad}}
\newcommand{ \Ad } {\operatorname{Ad}}
\newcommand{ \dist } {\operatorname{dist}}
\newcommand{ \rank } {\operatorname{rank}}
\newcommand{\Vol}{\operatorname{Vol}}
\newcommand{\dVol}{\operatorname{dVol}}
\newcommand{ \zitieren }[1]{ \hspace{-3mm} \cite{#1}}
\newcommand{ \pr }{\operatorname{pr}}
\newcommand{\diag}{\operatorname{diag}}
\newcommand{\Lagr}{\mathcal{L}}
\newcommand{\av}{\operatorname{av}}
\newcommand{ \floor }[1]{ \lfloor #1 \rfloor }
\newcommand{ \ceil }[1]{ \lceil #1 \rceil }
\newcommand{\Sym} {\operatorname{Sym}}
\newcommand{\bcirc}{ \ \bar{\circ} \ }
\newcommand{\conj}[1]{ \overline{ #1 } }
\newcommand{\sign}[1]{\operatorname{sign}(#1)}
\newcommand{\cone}{\operatorname{cone}}
\newcommand{\pbd}{\varphi_{bar}^{\delta}}

\newtheorem{theorem}{Theorem}[section]
\newtheorem{definition}[theorem]{Definition}
\newtheorem{example}[theorem]{Example}
\newtheorem{remark}[theorem]{Remark}
\newtheorem{lemma}[theorem]{Lemma}
\newtheorem{proposition}[theorem]{Proposition}
\newtheorem{corollary}[theorem]{Corollary}
\newtheorem{assumption}[theorem]{Assumption}
\newtheorem{acknowledgment}[theorem]{Acknowledgment}
\newtheorem{DefAndLemma}[theorem]{Definition and lemma}

\newcommand{\R}{\mathbb{R}}
\newcommand{\N}{\mathbb{N}}
\newcommand{\Z}{\mathbb{Z}}
\newcommand{\Q}{\mathbb{Q}}
\newcommand{\C}{\mathbb{C}}
\newcommand{\F}{\mathbb{F}}
\newcommand{\X}{\mathcal{X}}
\newcommand{\D}{\mathcal{D}}
\newcommand{\Cont}{\mathcal{C}}

\renewcommand{\labelenumi}{(\alph{enumi})}
\newtheorem{maintheorem}{Theorem}[]
\renewcommand*{\themaintheorem}{\Alph{maintheorem}}
\newtheorem*{theorem*}{Theorem}
\newtheorem*{corollary*}{Corollary}
\newtheorem*{remark*}{Remark}
\newtheorem*{example*}{Example}
\newtheorem*{question*}{Question}
\newtheorem*{definition*}{Definition}
\newtheorem{conjecture}[maintheorem]{Conjecture}
\newtheorem*{conjecture*}{Conjecture}

\begin{abstract}
In this paper, we prove the existence of an $SO(3)\times SO(8)$-invariant Einstein metric with positive scalar curvature on $S^{3}\times S^7$.
\end{abstract}

\maketitle

\section*{Introduction}

A Riemannian manifold $(M,g)$ is an Einstein manifold if its Ricci tensor satisfies $\Ric (g) = \lambda g$ for some $\lambda \in \R.$ There are many Einstein metrics on $S^3\times S^7$ that arise as product metrics. For example, with the proper scaling, one can take the product of the round metric on $S^3$ with an Einstein metric on $S^7$, such as the round metric on $S^7$, Jensen's homogeneous Einstein metric in \cite{JensenEinsteinMetricsOnPrincipalFibreBundles},  B\"ohm's cohomogenity one Einstein metrics in \cite{BohmInhomEinstein}, or Sasaki-Einstein metrics in \cite{BoyerGalickiKollarEinsteinMetricsOnSpheres, GhigiKollarKaehlerEinsteinOrbifolds, LiuSanoTasinIninitelySasakiEinsteinMetricsSpheres}. For a product manifold $M_1\times M_2$, we call a metric $g$ on $M_1\times M_2$ non-standard if it does not arise as a product metric. Non-standard $SO(d_1+1)\times SO(d_2+1)$-invariant Einstein metrics on $S^{d_1+1}\times S^{d_2}$ were first found by B\"ohm in \cite{BohmInhomEinstein} for $ d_1+d_2\leq 8$ where $d_1,d_2\geq 2$. In \cite{FoscoloHaskinsNearlyKaehler}, Foscolo-Haskins discovered a non-standard $SU(2)\times SU(2)$-invariant nearly K\"ahler Einstein metric on $S^3\times S^3$. In this paper, we prove the following theorem:

\begin{maintheorem}
\label{EinsteinMetricsOnProductSphereMainTheorem}
There exists a non-standard $SO(3)\times SO(8)$-invariant Einstein metric with positive scalar curvature on the product $S^3\times S^7$.\end{maintheorem}

More generally, our construction of a non-standard Einstein metric with positive scalar curvature works for doubly warped product metrics on $S^3\times M^7$, where $(M^7,g_E)$ is any complete Einstein $7$-manifold with positive scalar curvature.
 
Our Einstein metric on $S^3\times S^7$ is related to the non-round $SO(d_1+1)\times SO(d_2+1)$-invariant Einstein metrics on $S^{10}$ found by Nienhaus-Wink in \cite{EinsteinMetricsOnTheTenSpheres}. The relationship between Einstein metrics on products and Einstein metrics on spheres is actually quite common. The B\"ohm \cite{BohmInhomEinstein} and Foscolo-Haskins \cite{FoscoloHaskinsNearlyKaehler} Einstein metrics on products of spheres were all found during successful searches for non-round Einstein metrics on spheres.  In addition, Chi discovered a non-round $Sp(2)\times Sp(1)$-invariant Einstein metric on $S^8$ \cite{ChiPositiveEinsteinMetrics}. Also, using computer-assisted methods, Buttsworth-Hodgkinson proved in \cite{BHNumericalEinsteinMetricOnTwelveSphere} that there exists a non-round $SO(3)\times SO(10)$-invariant Einstein metric on $S^{12}$. In \cite{QiuShiWangComputerAssistedEinsteinMetrics}, using computer-assisted methods, Wang recovered known Einstein metrics on $S^{10}$ \cite{EinsteinMetricsOnTheTenSpheres} and $S^{12}$ \cite{BHNumericalEinsteinMetricOnTwelveSphere}, and proved the existence of a non-standard Einstein metric on $S^3\times S^7$ and non-round Einstein metrics on $S^{11}$, $S^{12}$ and $S^{13}$.
 
For our proof, we build up on the setup in \cite{EinsteinMetricsOnTheTenSpheres}. In \cite{EinsteinMetricsOnTheTenSpheres}, Nienhaus-Wink constructed $SO(d_1+1)\times SO(d_2+1)$-invariant Einstein metrics on $S^{10}$ by estimating the winding around the cone solution, i.e. the sine cone over the principal orbit $S^{d_1}\times S^{d_2}$. As the volume of the singular orbit collapses, solutions approach the cone solution along the trajectory corresponding to B\"ohm's Ricci flat metric on $\R^{d_1+1}\times S^{d_2}$ \cite{BohmNonCompactComhomOneEinstein}. We denote B\"ohm's Ricci flat metric as $\gamma^{RF}_1$, and specify the dimensions $d_1,d_2$ whenever needed. The crucial finding in this paper is that the winding of $\gamma^{RF}_1$ around the cone solution when $(d_1,d_2)=(2,7)$ is significantly larger than the winding of $\gamma^{RF}_1$ when $d_1+d_2=9$ and $(d_1,d_2)\not=(2,7)$. This allows us to detect a non-standard Einstein metric on $S^{3}\times S^{7}$ but not for other configurations, cf. Figure \ref{fig:fig1} and Figure \ref{fig:fig2} . We do clarify here that no rigorous statements are made on the winding angle of the cases $(d_1,d_2)\not=(2,7)$, as we only see numerically that the winding of $\gamma^{RF}_1$ in those cases is not sufficient to generate non-standard metrics on $S^{d_1+1}\times S^{d_2}$.

To find the $S^{10}$ metrics in \cite{EinsteinMetricsOnTheTenSpheres}, it suffices to have the winding angle of $\gamma^{RF}_1$ being at least $\arctan{\frac{9}{4}}$. In section 3 of this paper, we construct a forwardly preserved set to show that the winding angle of $\gamma^{RF}_1$ is actually $\arctan{\frac{9}{4}}+\pi$ for $(d_1,d_2)=(2,7)$. 

\begin{figure}[h]
    
    \centering
    \includegraphics[scale=0.21]{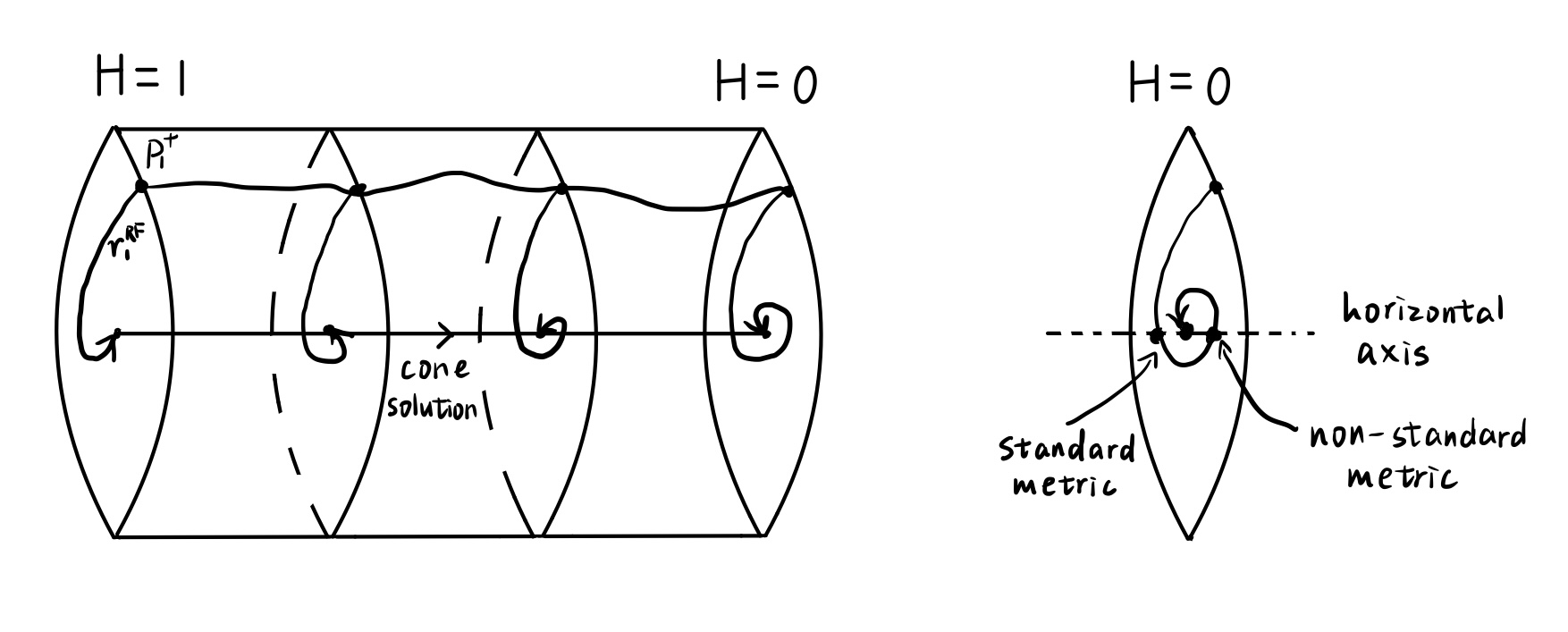}
    \caption{The winding in the case $(d_1,d_2)=(2,7)$.}
    \label{fig:fig1}
\end{figure}

\begin{figure}[h]
    
    \centering
    \includegraphics[scale=0.20]{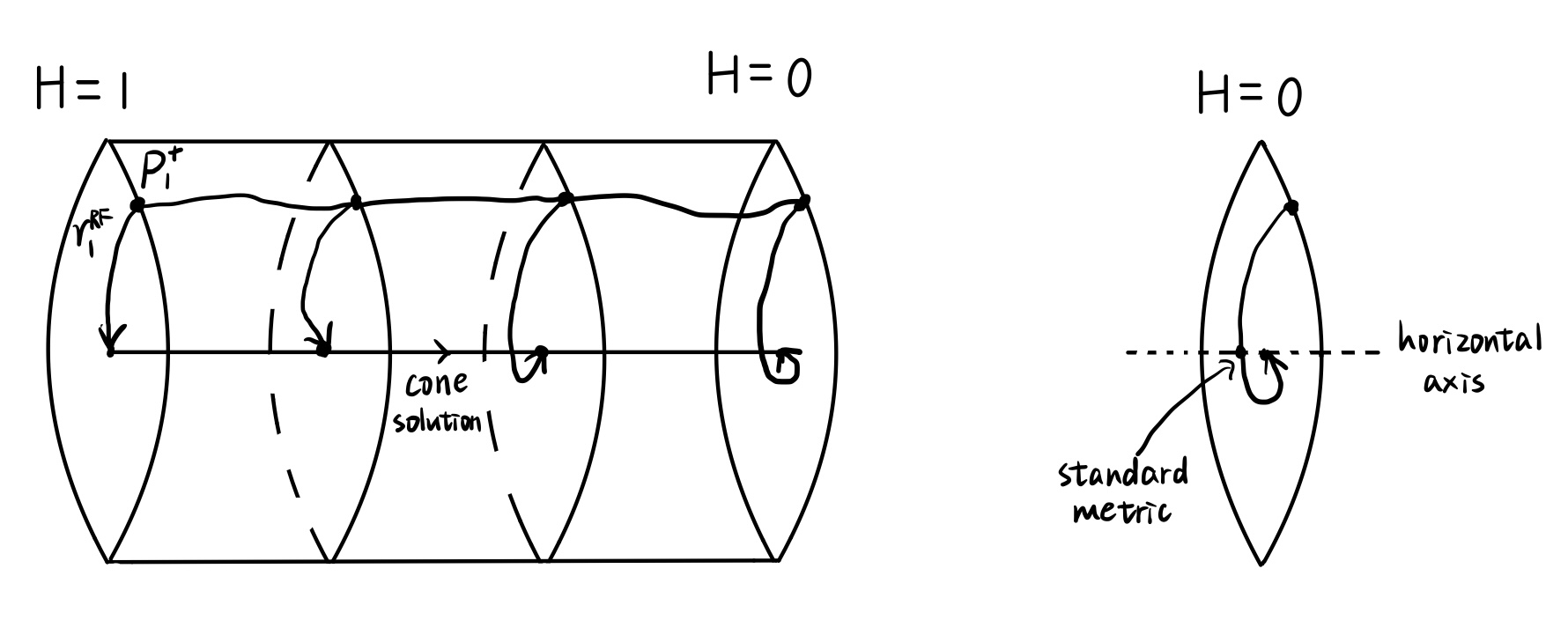}
    \caption{The winding in the cases $d_1+d_2=9$ with $(d_1,d_2)\not=(2,7)$.}
    \label{fig:fig2}
\end{figure}

\vspace{2mm}
\newpage
\textit{Structure.} In the preliminary section, we review the basic equations of cohomogenity one Einstein metrics on $S^{d_1+1}\times S^{d_2}$ that are invariant under the standard action of $SO(d_1+1)\times SO(d_2+1)$. In section 2, we set up the regular coordinates and prove some basic results, including a counting lemma for our Einstein metrics. In section 3, we prove that for $(d_1,d_2)=(2,7)$ B\"ohm's Ricci flat trajectory $\gamma^{RF}_1$ approaches the base point of the cone solution at the angle $\arctan{\frac{9}{4}}+\pi$. In section 4, we study the angle ODE along the cone solution and establish a useful barrier angle function. Finally in section 5, we prove Theorem \ref{EinsteinMetricsOnProductSphereMainTheorem}. \vspace{2mm}

\textit{Acknowledgments.} The author thanks his advisor Matthias Wink for pointing out that a non-standard Einstein metric on $S^3\times S^7$ was numerically detected during his joint work with Jan Nienhaus in \cite{EinsteinMetricsOnTheTenSpheres}. The author would also like to thank Patrick Donovan, Jia-Lin Hsu and Qiu Shi Wang for comments and suggestions.

\section{Preliminaries}
\label{SectionPreliminaries}
\subsection{Einstein metrics on the product of spheres} 
\label{SectionEinsteinMetricsOnSpheresIntro}
Let $(M,g)$ be an $(n+1)$-dimensional Riemannian manifold. $(M,g)$ is called Einstein if
\begin{equation*}
    \Ric = \lambda g.
\end{equation*} We consider $SO(d_1+1) \times SO(d_2+1)$-invariant Einstein metrics on the product of spheres $S^{d_1+1}\times S^{d_2}$, for $d_1, d_2 \geq 2$. This setup was considered by B\"ohm in \cite{BohmInhomEinstein} for $n=d_1+d_2\leq 8$ and by Nienhaus-Wink in \cite{EinsteinMetricsOnTheTenSpheres} for $n=d_1+d_2=9$. Away from the singular orbits, the metric is given by
\begin{align*}
    dt^2+g_t=dt^2+f_1^2(t) \ g_{S^{d_1}} + f_2^2(t) \ g_{S^{d_2}},
\end{align*}
where $g_{S^{d_i}}$ denotes the round metric on $S^{d_i}$. The shape operator and the $(1,1)$-Ricci curvature of the principal orbit $S^{d_1}\times S^{d_2}$ are
\begin{align*}
    L = \left( \frac{\Dot{f}_1}{f_1} \id_{d_1}, \frac{\Dot{f}_2}{f_2} \id_{d_2} \right), \ \ 
    r= \left( \frac{d_1 -1}{f_1^2} \id_{d_1}, \frac{d_2 -1}{f_2^2} \id_{d_2} \right).
\end{align*}
Eschenburg-Wang \cite{EschenburgWangInitialValueProblem} have proven that 
\begin{align}
\label{EvolutionMetric}
    \Dot{g}_t = 2 g_t L_t
\end{align}
and the Einstein equations are given by
\begin{align}
    \label{DerivativeL}
    - \Dot{L} -  \tr( L ) L + r & = \lambda \id, \\
    \label{DerivativeTrL}
    -\tr( \Dot{L} ) - \tr( L^2 ) & = \lambda, \\
    \Ric(X,N) & = 0, \nonumber
\end{align}
where $X$ is tangent to the principal orbit and $N$ is a horizontal lift of $\frac{\partial}{\partial t}$. Furthermore, any solution to \eqref{DerivativeL}, \eqref{DerivativeTrL} satisfies the constraint equation
\begin{align}
\label{GeneralConstraintEquation}
    \tr ( L^2 ) + \tr (r) - \tr(L)^2 = (n-1) \lambda.
\end{align}
In \cite[Theorem 2.3]{BohmInhomEinstein}, B\"ohm proved that for any $a>0$, there is a unique solution $(f_1, \Dot{f}_1, f_2, \Dot{f}_2)$ to \eqref{EvolutionMetric} -\eqref{DerivativeTrL} satisfying the initial condition
\begin{align}
\label{InitialSmoothCollapseSDOneFcoords}
    f_1(0)=0, \ \Dot{f}_1(0)=1, \ f_2(0)=a, \ \Dot{f}_2(0)=0,
\end{align} corresponding to the smooth collapse of $S^{d_1}$ at $t=0$. If the solution in addition satisfies the terminal condition  
\begin{align}
\label{TerminalSmoothCollapseSDOneFcoords}
    f_1(T)=0, \ \Dot{f}_1(T)=-1, \ f_2(T)=b, \ \Dot{f}_2(T)=0
\end{align}
for some $T>0$ and $b>0$, then $S^{d_1}$ smoothly collapses at $t=T$ and it induces a smooth complete Einstein metric on $S^{d_1+1}\times S^{d_2}$.

\section{Coordinate change}
\label{SectionCoordinateChanges}
We use the same coordinates as in \cite{EinsteinMetricsOnTheTenSpheres}. To avoid redundancy, we establish the coordinates in a rigorous but stream-lined manner. We refer curious readers to \cite[Section 2]{EinsteinMetricsOnTheTenSpheres} for a more geometric perspective of the coordinates. 

In section \ref{SectionZDeltaHCoordinates}, we establish the equivalent $(Z,\Delta,H)$-coordinates and prove some of its basic properties. In section \ref{SymmetryAndCounting}, we establish a counting lemma for complete $SO(d_1+1)\times SO(d_2+1)$-invariant Einstein metrics on $S^{d_1+1}\times S^{d_2}$.

\subsection{$(Z,\Delta,H)$-coordinates} 
\label{SectionZDeltaHCoordinates}
We fix $\lambda>0$ throughout the rest of this paper. Following \cite{ChiPositiveEinsteinMetrics}, we define 
\begin{align*}
    \mathcal{L}  = \frac{1}{\sqrt{ \tr(L)^2+n \lambda}}, \ \
    \frac{d}{ds}  = \mathcal{L} \ \frac{d}{dt}.
\end{align*}
The derivative with respect to $s$ will be denoted by prime $'$. Now, we define the $(Z,\Delta,H)$-coordinates by
\begin{align*}Z=(d_1-1)\frac{\mathcal{L}^2}{f_1^2}-(d_2-1)\frac{\mathcal{L}^2}{f_2^2}\ , \ \Delta=\mathcal{L}\frac{\Dot{f}_1}{f_1}-\mathcal{L}\frac{\Dot{f}_2}{f_2}\ , \ H=\mathcal{L} \ \tr(L).
\end{align*}
It follows from \eqref{DerivativeL} and \eqref{DerivativeTrL} that
\begin{align}
    \label{EinsteinODEinZDeltaH}
     Z^{'} & = \frac{2}{n} \Delta \left( d_1 d_2 Z \Delta H + \frac{d_1 d_2}{n} \Delta^2 - \frac{n-1}{n} + (d_1 -d_2) Z \right), \\
    \nonumber
    \Delta^{'} & = \Delta H \left( \frac{d_1 d_2}{n} \Delta^2 - \frac{n-1}{n} \right) +Z, \\
    \nonumber
    H^{'} & = - \frac{1-H^2}{n} \left( d_1 d_2 \Delta^2 +1 \right).
\end{align}

Same as in \cite{EinsteinMetricsOnTheTenSpheres}, we consider the set
\begin{align*}
   \mathcal{S} = \left\lbrace d_1 Z + \frac{d_1 d_2}{n} \Delta^2 \leq \frac{n-1}{n} \right\rbrace \cap \left\lbrace -d_2 Z + \frac{d_1 d_2}{n} \Delta^2 \leq \frac{n-1}{n} \right\rbrace \cap \left\lbrace H^2 \le 1 \right\rbrace.
\end{align*}
The set $\mathcal{S}$ is homeomorphic to the cylinder $\overline{D}_1(0)\times [-1,1]$. We shall see in a moment that solutions to the Einstein ODEs \eqref{EvolutionMetric}-\eqref{DerivativeTrL} can be recovered from solutions to the ODE system \eqref{EinsteinODEinZDeltaH} restricted to $\mathcal{S}$. First, note that we have some useful invariant sets.

\begin{remark}[cf. Remark 2.4 in \cite{EinsteinMetricsOnTheTenSpheres}]
\normalfont
    \label{SPreserved}
    The set $\mathcal{S}$ is compact. Its boundary and its interior are all invariant the flow of the ODE system \eqref{EinsteinODEinZDeltaH}. All of the following sets are also invariant 
\begin{align*}
    \bigg\{ d_1 Z + \frac{d_1 d_2}{n} \Delta^2 = \frac{n-1}{n}\bigg\} \ &,\ \bigg\{ -d_2 Z +\frac{d_1 d_2}{n} \Delta^2 = \frac{n-1}{n}\bigg\},\\ 
    \{-1<H<1\}\ ,\
    \{H&=1\}\  \text{and} \ \{H=-1\}.
\end{align*}
Furthermore, solutions in $\mathcal{S}$ exist for all times $s\in\R$: the solutions of the ODE have short time existence as the ODE is given by polynomials, so when restricted to the compact invariant set $\mathcal{S}$ they have long time existence.
\end{remark}

We also observe that $H$ is monotonically decreasing in $\{-1<H<1\}$.

\begin{proposition}
    \label{MonotonicityOfH}
    Let $(Z(s),\Delta(s),H(s))$ be a solution to the ODE system \eqref{EinsteinODEinZDeltaH} restricted to $\mathcal{S}$ with $H(s_0)\in (-1,1)$ at some $s_0\in\R$. Then $H(s)$ is strictly monotonically decreasing, and we have $\lim_{s\rightarrow-\infty}H(s)=1$ and $\lim_{s\rightarrow+\infty}H(s)=-1$.
\end{proposition}
\begin{proof}
    $H(s)$ being strictly monotonically decreasing follows immediately from the evolution equation for $H$ and the fact that the solution is contained in the invariant set $\{-1<H<1\}$.
    
    Next, we prove that $\lim_{s\rightarrow+\infty}H(s)=-1$. Assume for contradiction that's not the case. Then by boundedness and strict monotonicity of $H$, we must have  $\lim_{s\rightarrow+\infty}H(s)=h_\infty\in(-1,1)$ and therefore $H\equiv h_{\infty}$ on the $\omega$-limit set. As $\mathcal{S}$ is compact, the $\omega$-limit set of our solution is nonempty. However, since $H'$ is negative on $H\equiv h_\infty\in(-1,1)$, the $\omega$-limit set is not invariant under the ODE. This contradicts the invariance of $\omega$-limit sets.

    The proof for $\lim_{s\rightarrow-\infty}H(s)=1$ follows similarly, but we look at the $\alpha$-limit set instead of the $\omega$-limit set.
\end{proof}

In the next proposition, we demonstrate that we can recover solutions to the Einstein ODE system \eqref{EvolutionMetric}-\eqref{DerivativeTrL} from solutions to the ODE system \eqref{EinsteinODEinZDeltaH}.

\begin{proposition}[cf. Proposition 2.1 in \cite{EinsteinMetricsOnTheTenSpheres}]
\label{EquivalenceOfODEs}
A solution $(Z, \Delta,H)$ to the ODE system \eqref{EinsteinODEinZDeltaH} with  $H \in (-1,1)$, $\frac{n-1}{n}-d_1Z-\frac{d_1d_2}{n}\Delta^2>0$ and $\frac{n-1}{n}+d_2Z-\frac{d_1d_2}{n}\Delta^2>0$ at some $s_0 \in \R$ induces a solution to the Einstein ODEs \eqref{EvolutionMetric} - \eqref{DerivativeTrL} by setting
\begin{align*}
     t=t(s_0) +& \int_{s_0}^s \sqrt{\frac{1 - H^2}{n \lambda}} ( \tau ) d \tau,\\ 
     \frac{\Dot{f}_1}{f_1} = \frac{d_2}{n}\sqrt{\frac{n\lambda}{1-H^2}}\bigg(\Delta+\frac{1}{d_2}H\bigg)\ &, \
     \frac{\Dot{f}_2}{f_2} = \frac{d_1}{n}\sqrt{\frac{n\lambda}{1-H^2}}\bigg(\frac{1}{d_1}H-\Delta\bigg), \\
     f_1=\sqrt{\frac{(d_1-1)(1 - H^2)}{\lambda\bigg(\frac{n-1}{n}+d_2Z-\frac{d_1d_2}{n}\Delta^2\bigg)}}\ &, \ f_2=\sqrt{\frac{(d_2-1)(1 - H^2)}{\lambda\bigg(\frac{n-1}{n}-d_1Z-\frac{d_1d_2}{n}\Delta^2\bigg)}}.
\end{align*}
\end{proposition}
\begin{proof}
    Assume that we are given a solution in $(Z,\Delta,H)$ with the specified properties. Note that all the above functions defined in terms of $(Z,\Delta,H)$ are well defined, as $-1<H<1$, $\frac{n-1}{n}-d_1Z-\frac{d_1d_2}{n}\Delta^2>0$ and $\frac{n-1}{n}+d_2Z-\frac{d_1d_2}{n}\Delta^2>0$ are all invariant under the ODE system \eqref{EinsteinODEinZDeltaH} by Remark \ref{SPreserved}. Then it is easy to verify that the functions satisfy \eqref{EvolutionMetric}-\eqref{DerivativeTrL}.
\end{proof}

In the following remark, we list the fixed points of the Einstein ODE \eqref{EinsteinODEinZDeltaH} restricted to $\mathcal{S}$.

\begin{remark}[cf. Proposition 2.5 in \cite{EinsteinMetricsOnTheTenSpheres}]
\label{FixedPointsInHDeltaZ}
\normalfont
Restricted to $\mathcal{S}$, the fixed points of the Einstein ODE \eqref{EinsteinODEinZDeltaH} are given in $(Z, \Delta, H)$-coordinates by
\begin{enumerate}
    \item $p_1^+= \left( \frac{d_1-1}{d_1^2}, \frac{1}{d_1},1 \right)$ and $p_1^-= \left( \frac{d_1 -1}{d_1^2}, - \frac{1}{d_1}, -1 \right).$ These correspond to a smooth collapse of $S^{d_1}.$ 
    \item $p_2^+= \left( - \frac{d_2-1}{d_2^2}, - \frac{1}{d_2},1 \right)$ and $p_2^-= \left( - \frac{d_2 -1}{d_2^2}, \frac{1}{d_2}, -1 \right).$ These correspond to a smooth collapse of $S^{d_2}.$
    \item $\cone^{\pm}= \left( 0,0, \pm 1 \right).$ These correspond to the base points of the cone solution, cf. Remark \ref{SpecialSolutions}(b).
    \item $q_1^{\pm}= \left( 0, \mp \sqrt{\frac{n-1}{d_1 d_2}}, \pm 1 \right), q_2^{\pm}= \left( 0, \pm \sqrt{\frac{n-1}{d_1 d_2}}, \pm 1 \right).$ These correspond to singular solutions.   
\end{enumerate}
In particular, solutions emanating from $p^+_1$ correspond to solutions satisfying (\ref{InitialSmoothCollapseSDOneFcoords}), and solutions converging to $p^-_1$ correspond to solutions satisfying (\ref{TerminalSmoothCollapseSDOneFcoords}).
\end{remark}

We also compute the linearizations of two important critical points.

\begin{proposition}[cf. Proposition 2.6 in \cite{EinsteinMetricsOnTheTenSpheres}]
\label{LinearizationOfFixedPoints}
All of the fixed points of the Einstein ODE \eqref{EinsteinODEinZDeltaH} are hyperbolic. Furthermore:
\begin{enumerate}
        \item $p^+_1$ is a saddle. The linearization at $p^+_1$ has the eiganvalues $\frac{2}{d_1}$ of multiplicity 2 and and $-\frac{d_1-1}{d_1}$ with multiplicity 1. The eigenspace of $\frac{2}{d_1}$ is spanned by $(1+\frac{d_1-d_2}{d_1 n}, 1, 0 )$ and $(\frac{d_1-1}{d_1^2}, 0, 1)$. The unstable manifold of $p_1^+$ is $2$-dimensional and intersects $\partial \mathcal{S}$ transversally near $p_1^+$.
        \item $\cone^{+}$ is a saddle point of the ODE \eqref{EinsteinODEinZDeltaH}. When $n=9$, the linearization at $\cone^{+}$ has the eiganvalues $\frac{2}{9}$ with mutiplicity 1 and $-\frac{4}{9}$ with algebraic multiplicity $2$ and geometric multiplicity $1$. The eigenspace of $\frac{2}{9}$ is spanned by $(0,0,1)$ and the eigenspace of $-\frac{4}{9}$ is spanned by $(4,9,0)$.  Furthermore, $\cone^+$ is a sink for the 2-dimensional ODE obtained by restricting \eqref{EinsteinODEinZDeltaH} to the invariant set $\mathcal{S}\cap \{H=1\}$. 
    \end{enumerate}
\end{proposition}
\begin{proof}
    The above conclusion follows from a direct computation.
\end{proof}

\begin{definition}
    Denote $M_i^+$ (resp. $M_i^-$) as the intersection of the unstable manifold of $p_i^+$ (resp. the stable manifold of $p_i^-$) with $\mathcal{S}$.
\end{definition}

We make a remark on two important solutions in the $(Z,\Delta,H)$-coordinates.
\begin{remark}
\label{SpecialSolutions}
\normalfont
Here are two important solutions.
    \begin{enumerate}
        \item The unique trajectory in $M^+_1\cap\{H=1\}$ corresponds to B\"ohm's Ricci flat metric on $\R^{d_1+1}\times S^{d_2}$ in \cite{BohmNonCompactComhomOneEinstein}. We denote the trajectory by $\gamma^{RF}_1$. It emanates from $p^+_1$ and converges to $\cone^+$ by \cite[Proposition 2.15]{EinsteinMetricsOnTheTenSpheres}.  
        \item The cone solution $\cone^+(h):=(0,0,h)$, where $h(s)= -\tanh(s/n)$, corresponds to the sine-suspension over the principal orbit $S^{d_1}\times S^{d_2}$. It emanates from $\cone^+$ and converges to $\cone^-$.
    \end{enumerate}
\end{remark}

Around the cone solution, the solutions in $\mathcal{S}$ exhibit counterclockwise rotational behavior up to quadrants, discovered by Nienhaus-Wink in \cite{EinsteinMetricsOnTheTenSpheres}.

\begin{proposition}[Proposition 2.13 in \cite{EinsteinMetricsOnTheTenSpheres}]
\label{CounterclockwiseRotationUpToQuadrants}
Let $(Z,\Delta,H)$ be a solution to the Einstein ODE \eqref{EinsteinODEinZDeltaH} restricted to $\mathcal{S}$ that is not the cone solution. If it enters a quadrant in the $(Z,\Delta)$-plane, it either stays in that quadrant or goes into the next quadrant going counterclockwise around the cone solution.

\end{proposition}

\subsection{A counting lemma for $SO(d_1+1)\times SO(d_2+1)$-invariant Einstein metrics on $S^{d_1+1}\times S^{d_2}$}
\label{SymmetryAndCounting}
\begin{lemma}
Consider the map $\sigma:\mathcal{S}\rightarrow\mathcal{S}$ defined by $\sigma(Z,\Delta,H):=(Z,-\Delta,-H)$. Then $(Z(s),\Delta(s),H(s))$ solves the Einstein ODE \eqref{EinsteinODEinZDeltaH} restricted to $\mathcal{S}$ if and only if $\sigma(Z(-s),\Delta(-s),H(-s))$ solves the Einstein ODE \eqref{EinsteinODEinZDeltaH} restricted to $\mathcal{S}$. In particular, $\sigma(M^+_1)=M^-_1$.
\end{lemma}
\begin{proof}
    The $\sigma$-symmetry follows directly from the Einstein ODE \eqref{EinsteinODEinZDeltaH}. From a geometric point of view, one can also interpret the $\sigma$-symmetry as the freedom to choose the $t$-direction from one singular orbit to the other. The conclusion $\sigma(M^+_1)=M^-_1$ follows immediately from the definitions of stable/unstable manifolds and the fact that $\sigma(p^+_1)=p^-_1$.
\end{proof}

\begin{lemma}
\label{CountingLemma}
Let $d_1, d_2\ge 2$ be natural numbers. Modulo scaling and with respect to the standard action, $SO(d_1+1)\times SO(d_2+1)$-invariant Einstein metrics on $S^{d_1 +1}\times S^{d_2}$ are in one-to-one correspondence with points in $M_1^+ \cap \sigma(M^+_1) \cap \{H=0\}$. In particular, points in $M^+_1\cap\{H=0\}\cap\{\Delta=0\}$ induce $SO(d_1+1)\times SO(d_2+1)$-invariant Einstein metrics on $S^{d_1 +1}\times S^{d_2}$. 
\end{lemma}
\begin{proof}
    $SO(d_1+1)\times SO(d_2+1)$-invariant Einstein metrics on $S^{d_1 +1}\times S^{d_2}$ are exactly the solutions of the Einstein equations \eqref{EvolutionMetric}-\eqref{DerivativeTrL} that satisfy \eqref{InitialSmoothCollapseSDOneFcoords} and \eqref{TerminalSmoothCollapseSDOneFcoords}. By Proposition \ref{EquivalenceOfODEs} and Remark \ref{FixedPointsInHDeltaZ}, a solution of \eqref{EvolutionMetric}-\eqref{DerivativeTrL} with conditions \eqref{InitialSmoothCollapseSDOneFcoords} and \eqref{TerminalSmoothCollapseSDOneFcoords} corresponds uniquely to a trajectory in the $(Z,\Delta,H)$-subsystem emanating from $p^+_1$ and converging to $p^-_1$. Such trajectories form the set $M^+_1\cap M^-_1$ by the definition of stable and unstable manifolds.

    So, now it suffices to count the number of trajectories in the set $M^+_1\cap M^-_1$. By Proposition \ref{MonotonicityOfH}, any trajectory in $M^+_1\cap M^-_1$ is uniquely determined by its intersection point with $\{H=0\}$. Then it suffices to count the cardinality of $M^+_1\cap M^-_1\cap\{H=0\}$. Finally, note that by Lemma 2.9, we have $M^+_1\cap M^-_1\cap\{H=0\}=M^+_1\cap \sigma(M^+_1)\cap \{H=0\}\supseteq M^+_1\cap\{H=0\}\cap\{\Delta=0\} $.
\end{proof}

\begin{remark}
    \label{SymmetricSolutions}
    \normalfont 
    Points in $M^+_1\cap\{H=0\}\cap\{\Delta=0\}$ induce metrics that are symmetric in the sense of B\"ohm \cite[Section 3]{BohmInhomEinstein}, as $H(s_0)=\Delta(s_0)=0$ iff $\dot{f}_1(t(s_0))=\dot{f}_2(t(s_0))=0$.
\end{remark}

The following remark ensures that any newly found metric in our setting must be non-standard. Recall that a metric $g$ on a product manifold is called non-standard if it does not arise as a product metric.
    
\begin{remark}
   \label{ProductOfRoundOrNonstandard}
   \normalfont
   Any $SO(3)\times SO(8)$-invariant Einstein metric on $S^{3}\times S^{7}$ is either a product of the scaled round metrics or a non-standard Einstein metric. To see this, we consider any $SO(3)\times SO(8)$-invariant Einstein metric $g=dt^2+f_1^2(t)g_{S^2}+f^2_2(t)g_{S^7}$ on $S^{3}\times S^{7}$ that also arises as a product metric, and show that it must be a product of the scaled round metrics. Since $g$ is a product metric that is Einstein, its factors are also Einstein. Note that any Einstein metric on $S^3$ must be the scaled round metric as the Weyl curvature tensor vanishes in dimension $3$. So we must have $f_1(t)=c\sin(c^{-1}t)$ for some $c>0$. Plugging $f_1(t)=c\sin(c^{-1}t)$ into the Einstein ODEs \eqref{DerivativeL}-\eqref{DerivativeTrL} immediately implies that $f_2$ must be constant. Therefore $g$ is a product of the scaled round metrics.
\end{remark}

\section{The converging angle of $\gamma^{RF}_1$}
\label{SectionProofOfconvergingangle}

Recall that $\gamma^{RF}_1$ denotes the unique trajectory in $M^+_1\cap\{H=1\}$, i.e. B\"ohm's Ricci flat metric on $\R^{d_1+1}\times S^{d_2}$\cite{BohmNonCompactComhomOneEinstein}. In this section, we prove the most important fact in this paper: when $(d_1,d_2)=(2,7)$, $\gamma^{RF}_1$ converges to $\cone^+$ at the angle $\arctan{\frac{9}{4}}+\pi$ in the $\{H=1\}$ plane.

Throughout this section, we fix $(d_1,d_2)=(2,7)$. When $(d_1,d_2)=(2,7)$, the ODE in the invariant subspace $\{H=1\}$ is given by
\begin{align}
    \label{ODEinH=1}
    Z^{'}&=\frac{2}{9}\Delta\bigg(14Z\Delta+\frac{14}{9}\Delta^2-\frac{8}{9}-5Z\bigg)\\
    \nonumber
    \Delta^{'}&=\Delta\bigg(\frac{14}{9}\Delta^2-\frac{8}{9}\bigg)+Z.
\end{align}

First, we construct a forwardly preserved set $\mathcal{A}$ in the subsystem $\mathcal{S}\cap\{H=1\}$, cf. Figure \ref{fig:fig3}.
\begin{figure}[h]
    \centering
    \includegraphics[scale=0.18]{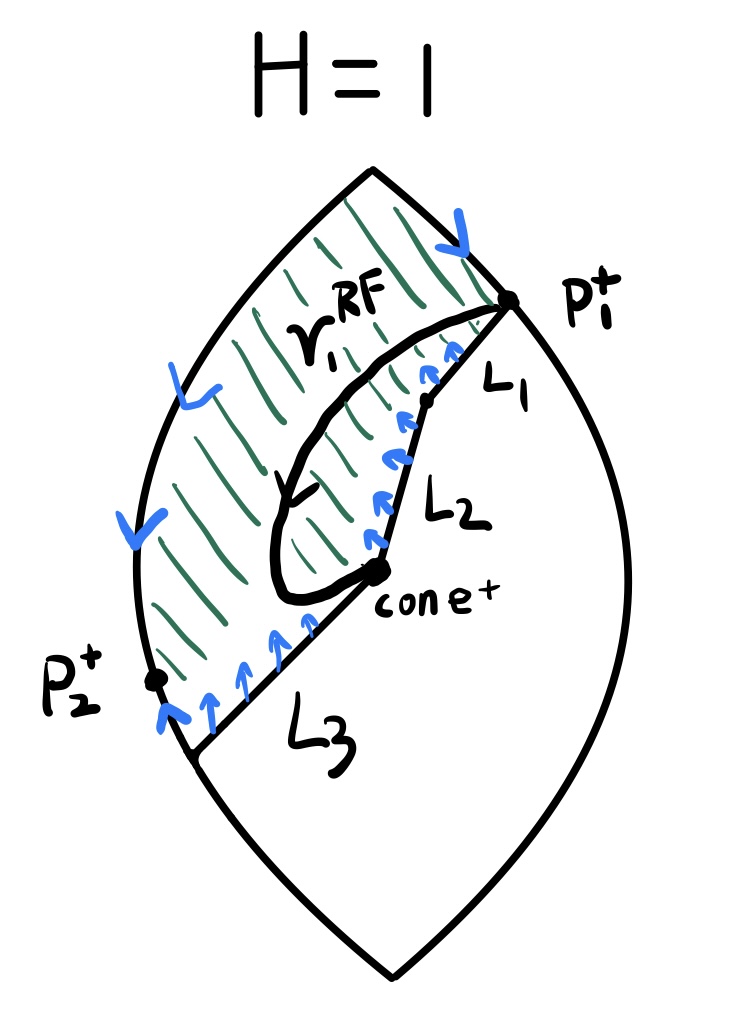}
    \caption{The invariant set $\mathcal{A}$: $\mathcal{A}$ corresponds to the dashed region; the blue arrows along the boundary of $\mathcal{A}$ describe the direction of the ODE vector field; $L_1$ is the boundary part with $\Delta-\frac{9}{5}(Z-\frac{1}{4})-\frac{1}{2}= 0$; $L_2$ is the boundary part with $\Delta-\frac{23}{10}Z= 0$; $L_3$ is the boundary part with $\Delta-\frac{9}{4}Z= 0$.}
    \label{fig:fig3}
\end{figure}

\begin{lemma}
\label{BarrierLemma}
Fix $(d_1,d_2)=(2,7)$. The set $\mathcal{A}:=\mathcal{S}\cap\{H=1\}\cap\bigg(\{\Delta-\frac{9}{5}(Z-\frac{1}{4})-\frac{1}{2}\geq 0,\frac{1}{10}\leq Z\leq\frac{1}{4}\}\cup\{\Delta- \frac{23}{10}Z\geq 0,0\leq Z\leq \frac{1}{10}\}\cup\{\Delta- \frac{9}{4}Z\geq 0,Z\leq 0\}\bigg)$ is forwardly preserved by the Einstein ODE in $\{H=1\}$ \eqref{ODEinH=1}, i.e. if a trajectory enters $\mathcal{A}$ at $s_0$ it will not leave $\mathcal{A}$ for any $s>s_0$.
\end{lemma}

\begin{proof}
    Since $\mathcal{S}\cap\{H=1\}$ is invariant under the ODE, it suffices to show that trajectories cannot exit through any of the three boundary parts $\{\Delta-\frac{9}{5}(Z-\frac{1}{4})-\frac{1}{2}= 0,\frac{1}{10}\leq Z\leq\frac{1}{4}\}$, $\{\Delta- \frac{23}{10}Z=0,0\leq Z\leq \frac{1}{10}\}$, or $\{\Delta- \frac{9}{4}Z=0,Z\leq 0\}$. 

    First we look at the segment $\{\Delta-\frac{9}{5}(Z-\frac{1}{4})-\frac{1}{2}= 0,\frac{1}{10}\leq Z\leq\frac{1}{4}\}$. Note that
    \begin{align*}
        &5\bigg(\Delta-\frac{9}{5}(Z-\frac{1}{4})-\frac{1}{2}\bigg)'\bigg|_{ \Delta-\frac{9}{5}(Z-\frac{1}{4})-\frac{1}{2}=0}\\
        =&5\Delta'-9Z'|_{ \Delta-\frac{9}{5}(Z-\frac{1}{4})-\frac{1}{2}=0}\\
        =&5\bigg(\frac{9}{5}Z+\frac{1}{20}\bigg)\bigg[\frac{14}{9}\bigg(\frac{9}{5}Z+\frac{1}{20}\bigg)^2-\frac{8}{9}\bigg]+5Z\\
        &-2\bigg(\frac{9}{5}Z+\frac{1}{20}\bigg)\bigg[14Z\bigg(\frac{9}{5}Z+\frac{1}{20}\bigg)+\frac{14}{9}\bigg(\frac{9}{5}Z+\frac{1}{20}\bigg)^2-\frac{8}{9}-5Z\bigg]\\
        =&-\frac{7938}{125}Z^3+\frac{3807}{250}Z^2+\frac{693}{1000}Z-\frac{531}{4000}.
    \end{align*}
In Lemma \ref{LemmaA.1} in the Appendix, we prove that the above polynomial is positive for for $Z\in [\frac{1}{10},\frac{1}{4}),$ and when $Z=\frac{1}{4}$ on the segment we are at the fixed point $p^+_1$. Therefore we see that trajectories cannot exit our set through the segment $\{\Delta-\frac{9}{5}(Z-\frac{1}{4})-\frac{1}{2}= 0,\frac{1}{10}\leq Z\leq\frac{1}{4}\}$.

Similarly for the other two boundary parts, we have 
\begin{align*}
    &(\Delta-kZ)'|_{\Delta-kZ=0}\\
    =&-\frac{2k^2}{9}Z\bigg(14kZ^2+\frac{14k^2}{9}Z^2-\frac{8}{9}-5Z\bigg)+kZ\bigg(\frac{14k^2}{9}Z^2-\frac{8}{9}\bigg)+Z\\
    =&Z\bigg(-\frac{28k^4+126k^3}{81}Z^2+\frac{10k^2}{9}Z+\frac{16k^2-72k+81}{81}\bigg).
\end{align*}
By Lemma \ref{LemmaA.2} in the Appendix, we see that when $k=\frac{23}{10}$ the above polynomial is positive for $Z\in(0,\frac{1}{10}]$, and in addition when $Z=0$ on the segment we are at the fixed point $\cone^+$. Therefore we see that trajectories cannot exit our set through the segment $\{\Delta-\frac{23}{10}Z=0, 0\leq Z\leq\frac{1}{10}\}$. Similarly, by Lemma \ref{LemmaA.3} in the Appendix, we see that when $k=\frac{9}{4}$ the above polynomial is positive for $Z<0$, so we see that trajectories also cannot exit our set through $\{\Delta-\frac{9}{4}Z=0, Z\leq0\}$.

Therefore we have proven that trajectories in our set cannot exit through any part of its boundary, so the set $\mathcal{A}$ is forwardly preserved by the ODE \eqref{ODEinH=1}.
\end{proof}

Now we prove that $\gamma^{RF}_1$ converges to $(0,0,1)$ at the angle of $\arctan\frac{9}{4}+\pi$ in the $\{H=1\}$ plane.

\begin{lemma}
\label{ConvergingAngleofRF}
Fix $(d_1,d_2)=(2,7)$. Let $\varphi^{RF}_1$ be the unique continuous function such that $\gamma^{RF}_1(s)=|\gamma^{RF}_1(s)|\begin{pmatrix}\cos\varphi_1^{RF}(s)\\ \sin \varphi^{RF}_1(s)\end{pmatrix}$ and $\lim\limits_{s\rightarrow-\infty}\varphi^{RF}_1(s)=\arctan2$. Namely, $\varphi^{RF}_1$ is the angle function on $\gamma^{RF}_1$. Then we have the converging angle $\lim\limits_{s\rightarrow\infty}\varphi^{RF}_1(s)=\arctan\frac{9}{4}+\pi$.
\end{lemma}
\begin{proof}
In $\{H=1\}$ the stable eigenspace of $\cone^+$ is spanned by the vector $(4,9)$, so by \cite[Chapter 15]{CoddingtonLevinsonODEs} we must have $\lim_{s\rightarrow\infty}\varphi^{RF}(s)=\arctan\frac{9}{4}+m\pi$ for some $m\in \mathbb{Z}$. The strategy here is to use the forwardly preserved set $\mathcal{A}$ constructed in Lemma \ref{BarrierLemma} to rule out all the cases except for $m=1$.

First, we want to show that $\gamma^{RF}_1$ must enter $\mathcal{A}$. Recall that $\gamma^{RF}_1$ emanates from $p^+_1$, so by the linearization at $p^+_1$ and the unstable version of Theorem 4.5 in \cite[Chapter 13]{CoddingtonLevinsonODEs}, there is some $\varepsilon>0$ such that
\begin{align*}
    (Z(\gamma^{RF}_1(s)),\Delta(\gamma^{RF}_1(s)))=\bigg(\frac{1}{4},\frac{1}{2}\bigg)+e^{s}\bigg(-\frac{13}{18},-1\bigg)+O\bigg(e^{(1+\varepsilon)s}\bigg).
\end{align*}
Therefore for all $s<0$ with sufficiently large absolute value, we have
\begin{align*}
    &\Delta(\gamma^{RF}_1(s))-\frac{9}{5}\bigg(Z(\gamma^{RF}_1(s))-\frac{1}{4}\bigg)-\frac{1}{2}\\
    =&\frac{1}{2}-e^s-\frac{9}{5}\bigg(\frac{1}{4}-\frac{13}{18}e^s-\frac{1}{4}\bigg)-\frac{1}{2}+O\bigg(e^{(1+\varepsilon)s}\bigg)\\
    =&\frac{3}{10}e^s+O\bigg(e^{(1+\varepsilon)s}\bigg)\\
    >&0,
\end{align*}
and consequently $\gamma^{RF}_1(s)\in \mathcal{A}$. So by Lemma \ref{BarrierLemma}, we see that $\gamma^{RF}_1(\R)\subseteq\mathcal{A}$.

Now, we rule out the approaching angle $\arctan\frac{9}{4}$. Assume for contradiction that $\lim_{s\rightarrow\infty}\varphi^{RF}_1(s)=\arctan\frac{9}{4}$. Since $\arctan{\frac{9}{4}}$ is an angle in the first quadrant that is smaller than $\arctan{\frac{23}{10}}$, for sufficiently large $s>0$ we have
\begin{align*}
    \ Z(\gamma^{RF}_1(s))>0 \ , \ \Delta(\gamma^{RF}_1(s))>0 \ , \ \varphi^{RF}_1(s)<\arctan\frac{23}{10} \ , 
\end{align*}
and therefore
\begin{align*}
    \Delta(\gamma^{RF}_1(s)) < \frac{23}{10}Z(\gamma^{RF}_1(s)))
\end{align*}
i.e. $\gamma^{RF}_1(s)\notin\mathcal{A}$. This directly contradicts the fact $\gamma^{RF}_1(\R)\subseteq\mathcal{A}$. So we cannot have $\lim_{s\rightarrow\infty}\varphi^{RF}_1(s)=\arctan\frac{9}{4}$.

Next, we rule out approaching angles $\arctan\frac{9}{4}+l\pi$ for any integer $l\geq 2$. Assume for contradiction that $\lim_{s\rightarrow\infty}\varphi^{RF}_1(s)=\arctan\frac{9}{4}+l\pi$ for some $l\geq 2$. Then $\lim_{s\rightarrow\infty}\varphi^{RF}_1(s)>\frac{3}{2}\pi$ and by intermediate value theorem we see that $\gamma^{RF}_1$ must have entered the fourth quadrant $\{Z> 0,\Delta<0\}$ at some time. However, note that this contradicts the fact that $\gamma^{RF}_1(\mathbb{R})\subseteq\mathcal{A}$ and $\mathcal{A}\cap\{Z> 0,\Delta<0\}=\varnothing$ (one can easily verify that the intersection is indeed empty). Therefore, we cannot have $\lim_{s\rightarrow\infty}\varphi^{RF}_1(s)=\arctan\frac{9}{4}+l\pi$ for any $l\geq 2$.

Finally, since $\lim\limits_{s\rightarrow-\infty}\varphi^{RF}_1(s)=\arctan2$, the approaching angles $\arctan\frac{9}{4}+l\pi$ for integers $l<0$ are immediately ruled out by the counterclockwise rotation behavior up to quadrants (Proposition \ref{CounterclockwiseRotationUpToQuadrants}). Therefore, we have ruled out all the approaching angles except for $\arctan{\frac{9}{4}}+\pi$, and we must have $\lim_{s\rightarrow\infty}\varphi^{RF}_1(s)=\arctan\frac{9}{4}+\pi$.
\end{proof}
\section{The angle ODE along the cone solution}
\label{SectionODEestimates}

Similar to section 4 in \cite{EinsteinMetricsOnTheTenSpheres}, to estimate the winding along the cone solution $\cone^+(h)=(0,0,h)$, we look at the linearized angle ODE along it and establish our version of angle barrier function. Throughout this section, we assume $n=9$. The linearization along the cone solution is given by 
\begin{align}
\label{LinearizationAlongConeSolution}
D(ODE)_{\cone^+(h)} =\begin{pmatrix}
0 & - \frac{2(n-1)}{n^2} & 0 \\
1 & - \frac{n-1}{n} h & 0 \\
0 & 0 & \frac{2}{n} h
    \end{pmatrix}.
\end{align}
If
\begin{align*}
X(h) = |X(h)| \begin{pmatrix} \cos{\varphi(h)} \\ \sin{\varphi(h)}\end{pmatrix}
\end{align*}
solves \eqref{LinearizationAlongConeSolution},
\begin{align*}
\frac{d}{ds}X(h(s)) = \begin{pmatrix}
0 & - \frac{2(n-1)}{n^2} \\
1 & - \frac{n-1}{n} h(s)
    \end{pmatrix} X(h(s)),
\end{align*}
then $\varphi$ solves the associated angle ODE
\begin{align}
    \label{AngleODEeq}
    \frac{d}{ds}\varphi(h(s)) = & \ \cos^2(\varphi) + \frac{2(n-1)}{n^2}\sin^2(\varphi) - \frac{n-1}{n}h(s)\sin(\varphi)\cos(\varphi)
\end{align}
where $h(s)=-\tanh(s/n)$. \vspace{2mm}

We consider a perturbed version of $\eqref{AngleODEeq}$ and construct a barrier solution $\varphi_{bar}.$

\begin{lemma}
\label{PBDLemma}
Let $n=9$. There exist $\varepsilon_{bar}>0$ and $\delta_0>0$ such that for all $\delta \in [0, \delta_0)$ the solution $\varphi_{bar}^{\delta} \colon \R \to \R$ of the initial value problem
\begin{align}
\label{DeltaAngleODE}
    \frac{d}{ds}\varphi(h(s)) & = \cos^2(\varphi) + \frac{2(n-1)}{n^2}\sin^2(\varphi) - \frac{n-1}{n}h(s)\sin(\varphi)\cos(\varphi) - \delta, \\
    \nonumber
    \varphi_{bar}^{\delta}(0) & =\frac{5}{2}\pi + \delta,
\end{align}
where $h(s)=-\tanh(s/n),$ satisfies
\begin{align*}
\varphi_{bar}^{\delta}(h(s)) < \arctan(\frac{9}{4})+\pi - \varepsilon_{bar}
\end{align*}
for all $s \in \R$ with $h(s)>0.995$. 
\end{lemma}

\begin{proof}
In \cite[Lemma 4.3]{EinsteinMetricsOnTheTenSpheres}, it was proven that if $\varphi_{bar}^{\delta}(h(0))  =\frac{3}{2}\pi + \delta$ then $\varphi_{bar}^{\delta}(h(s)) < \arctan(\frac{9}{4}) - \varepsilon_{bar}$ for all $s \in \R$ with $h(s)>0.995$. Our conclusion follows immediately from the fact that $\varphi(h(s))$ solves \eqref{DeltaAngleODE} if and only if $\varphi(h(s))+\pi$ solves \eqref{DeltaAngleODE}.
\end{proof}

\section{Proof of the main theorem}
\label{SectionProofOfMainTheorem}
In this section, we prove Theorem \ref{EinsteinMetricsOnProductSphereMainTheorem}. Throughout this section, we assume $d_1=2$ and $d_2=7$.
\begin{proof}
    By Lemma \ref{CountingLemma}, it suffices to show that there are at least two points in $M^+_1\cap\{H=0\}\cap\{\Delta=0\}$, one of them corresponding to the product of the scaled rounds metrics and the other point corresponding to a non-standard Einstein metric on $S^3\times S^7$. Therefore, we want to show that the curve $M^+_1\cap\{H=0\}$ winds such that it crosses the horizontal axis $\{\Delta=0\}$ twice.

    We first give a global parameterization of  $M^+_1\cap\{H^2<1\}$. We define the parameterization $\gamma:(-1,1)\times [-1,0)\rightarrow M^+_1\cap\{H^2<1\}$ as follows: for any $h\in (-1,1)$ and $\tau\in [-1,0)$ , let $\gamma(h,\tau)$ be the unique intersection point of $\{H=h\}$ and the trajectory starting from $p^+_1$ with the initial tangent direction $\tau\cdot(\frac{d_1-1}{d_1^2},0,1)+(-\frac{d_1^2+2d_1-1}{d_1^2+d_1}\tau-1)\cdot(1+\frac{d_1-d_2}{d_1n},1,0)$.  By the monotinicity of $H$ (Proposition \ref{MonotonicityOfH}) and continuous dependence on the initial conditions, we see that the parameterization is well-defined and continuous. When $\tau=-1$, the curve $h\rightarrow\gamma(h,-1)$ is the unique trajectory in the boundary part $M^+_1\cap\partial\mathcal{S}\cap\{H^2<1\}$. When $\tau\nearrow 0$, the curves $h\rightarrow \gamma(h,\tau)$ initially follow the Ricci flat trajectory $\gamma^{RF}_1$ and would then stay close to the cone solution until $H=0$ \cite[Remark 2.18]{EinsteinMetricsOnTheTenSpheres}. Note that  $\tau\mapsto\gamma(0,\tau)$ is a parameterization of the curve $M^+_1\cap\{H=0\}$.

    As we want to study the winding, we define an angle function $\varphi$ on $\gamma$ by declaring $\varphi:(-1,1)\times [-1,0)\rightarrow\R$ to be the unique continuous function such that $\gamma(h,\tau)=|\gamma(h,\tau)|\begin{pmatrix}\cos\varphi(h,\tau)\\ \sin \varphi(h,\tau)\end{pmatrix}$ and $\lim_{h\rightarrow 1}\varphi(h,-1)=\arctan2$. Using the boundary singular solution $q_2^+(h)=(0,\sqrt{\frac{n-1}{d_1d_2}},h)$ \cite[Prop 2.8]{EinsteinMetricsOnTheTenSpheres} as a barrier and the counterclockwise rotation behavior up to quadrants (Proposition \ref{CounterclockwiseRotationUpToQuadrants}), Nienhaus-Wink have shown that $\varphi(0,-1)\in(0,\frac{\pi}{2})$ \cite[Section 6]{EinsteinMetricsOnTheTenSpheres}. Therefore to show that $M^+_1\cap\{H=0\}$ crosses the horizontal axis $\{\Delta=0\}$ at least twice, we just have to prove that $\varphi(0,\tau^*)>2\pi$ for some $\tau^*>-1$ and the conclusion would follow immediately from applying the intermediate value theorem to $\varphi(0,\cdot)$ with intermediate values $\pi$ and $2\pi$.

    The idea is to build a suitable local barrier along the cone solution. In \cite{EinsteinMetricsOnTheTenSpheres}, there is a $c$-parameter family of neighborhoods $I_{c}$ of the cone solution, such that every $I_c$ is forwardly preserved by the Einstein ODE \eqref{EinsteinODEinZDeltaH} until $H\geq 0$ \cite[Remark 2.14 and 2.18]{EinsteinMetricsOnTheTenSpheres}. Furthermore, for every $r>0$ we can find a parameter $c(r)$ such that $I_{c(r)}$ is contained in the $r$-tube along the cone solution. Picking a small $r_0>0$ and arguing the same way as in \cite[Prop. 5.1]{EinsteinMetricsOnTheTenSpheres}, we have the following maximum principle: if $\gamma(h^*,\tau^*)\in I_{c(r_0)}$ and $\varphi(h^*,\tau^*)\geq\varphi^{\delta}_{bar}(h^*)$ for some $(h^*,\tau^*)$, then $\varphi(h,\tau^*)\geq \varphi^{\delta}_{bar}(h)$ for all $0\leq h\leq h^*$, where $\varphi^{\delta}_{bar}$ is the angle barrier function in Lemma \ref{PBDLemma}. In particular we would have $\varphi(0,\tau^*)\geq\varphi^{\delta}_{bar}(0)=\frac{5}{2}\pi+\delta>2\pi$. So it's only left for us to find $(h^*,\tau^*)$ such that $\gamma(h^*,\tau^*)\in I_{c(r_0)}$ and $\varphi(h^*,\tau^*)\geq\varphi^{\delta}_{bar}(h^*)$.

    By Lemma \ref{ConvergingAngleofRF}, we recall that $\lim_{s\rightarrow\infty}\varphi^{RF}(s)=\arctan{\frac{9}{4}}+\pi$. Also recall that $\lim_{s\rightarrow\infty}\gamma^{RF}(s)=\cone^+\in \operatorname{int}(I_{c(r_0)})$. So there must exist $s_0$ such that $\varphi^{RF}(s_0)>\arctan{\frac{9}{4}}+\pi-\frac{\varepsilon_{bar}}{2}$ and $\gamma^{RF}(s_0)\in \operatorname{int}(I_{c(r_0)})$, where $\varepsilon_{bar}>0$ is the $\varepsilon_{bar}$ in Lemma \ref{PBDLemma}. As $M^+_1=\gamma^{RF}_1(\R)\cup (M^+_1\cap\{H^2<1\})$, by continuous dependence of initial conditions we can certainly pick some $p^*\in M^+_1\cap\{H^2<1\}$ such that $p^*$ is arbitrarily close to $\gamma^{RF}(s_0)$. Since $\gamma$ is a global parameterization of $M^+_1\cap\{H^2<1\}$, there exists $(h^*,\tau^*)$ such that $\gamma(h^*,\tau^*)=p^*$. Choosing $p^*$ close enough to $\gamma^{RF}_1(s_0)$, we can ensure that $h^*>0.995$, $\varphi(h^*,\tau^*)>\varphi^{RF}(s_0)-\frac{\varepsilon_{bar}}{2}$ and $\gamma(h^*,\tau^*)\in I_{c(r_0)}$. Finally, note that
    \begin{align*}
        \varphi(h^*,\tau^*)&>\varphi^{RF}(s_0)-\frac{\varepsilon_{bar}}{2}\\
        &>\arctan{\frac{9}{4}}+\pi-\frac{\varepsilon_{bar}}{2}-\frac{\varepsilon_{bar}}{2}\\
        &=\arctan{\frac{9}{4}}+\pi-\varepsilon_{bar}\\
        &>\varphi^{\delta}_{bar}(h^*),
    \end{align*}
    where the last inequality follows from Lemma \ref{PBDLemma}. Therefore, we have found our $(h^*,\tau^*)$.
\end{proof} 

\appendix
\section{Appendix}

In this section, we prove some useful elementary lemmas.

\begin{lemma}
\label{LemmaA.1}
 The polynomial
\begin{align*}
    p(Z)&=-\frac{7938}{125}Z^3+\frac{3807}{250}Z^2+\frac{693}{1000}Z-\frac{531}{4000}
\end{align*}
has three real roots $r_1<r_2<r_3$, where $r_1<0<r_2<\frac{1}{10}<\frac{1}{4}=r_3$. In particular, we have $p(Z)>0$ for $Z\in [\frac{1}{10},\frac{1}{4})$.
\end{lemma}

\begin{proof}
We first locate the roots. The polynomial $p$ has the same roots as its rescaled version
\begin{align*}
    q(Z):&=36000000\cdot p(z)\\
    &=-2286144000Z^3+548208000Z^2+24948000Z-4779000.
\end{align*}
So it suffices to locate the roots of $q.$ We easily see that $\lim\limits_{Z\rightarrow-\infty}q(Z)=\infty$ and $\lim\limits_{Z\rightarrow\infty}q(Z)=-\infty$, and also we have
\begin{align*}
    q(0)=-4779000<0\ ,\ q(\frac{1}{10})=911736>0\ \text{and} \ q(\frac{1}{4})=0.
\end{align*}
So by intermediate value theorem, we see that $r_1<0<r_2<\frac{1}{10}<\frac{1}{4}=r_3$ and consequently $p(Z)>0$ for $Z\in [\frac{1}{10},\frac{1}{4})$.
\end{proof}

\begin{lemma}
\label{LemmaA.2}
Set $k:=\frac{23}{10}$.  The quadratic polynomial
\begin{align*}
    h(Z)=-\frac{28k^4+126k^3}{81}Z^2+\frac{10k^2}{9}Z+\frac{16k^2-72k+81}{81}
\end{align*}
has real roots $r_1,r_2$ such that $r_1<0<\frac{10^2}{23^2}<r_2$. In particular, we have $h(Z)>0$ for $Z\in [0,\frac{1}{10}]$.
\end{lemma}

\begin{proof}
We first locate the roots. Note that $\lim\limits_{Z\rightarrow-\infty}h(Z)=\lim\limits_{Z\rightarrow\infty}h(Z)=-\infty$, also we have
\begin{align*}
    h(0)&=\frac{16k^2-72k+81}{81}=\frac{16}{81}(k-\frac{9}{4})^2=\frac{16}{81}(\frac{23}{10}-\frac{9}{4})^2>0\\
    h(k^{-2})&=-\frac{28+126k^{-1}}{81}+\frac{10}{9}+\frac{16}{81}(k-\frac{9}{4})^2\\
    &\geq-\frac{28+\frac{126\cdot10}{23}}{81}+\frac{10}{9}\\
    &>-\frac{28+55}{81}+\frac{10}{9}\\
    &>0.
\end{align*} 
So by intermediate value theorem, we see that $r_1<0<k^{-2}=\frac{10^2}{23^2}<r_2$ and consequently $h(Z)>0$ for $Z\in [0,\frac{1}{10}]$.
\end{proof}

\begin{lemma}
\label{LemmaA.3}
Set $k:=\frac{9}{4}$. Consider the quadratic polynomial
\begin{align*}
    h(Z)=-\frac{28k^4+126k^3}{81}Z^2+\frac{10k^2}{9}Z+\frac{16k^2-72k+81}{81}.
\end{align*}
We have $h(Z)<0$ for any $Z<0$.
\end{lemma}

\begin{proof}
Note that when $k=\frac{9}{4}$ we have
\begin{align*}
    h(Z)=-\frac{28k^4+126k^3}{81}Z^2+\frac{10k^2}{9}Z=Z\bigg(-\frac{28k^4+126k^3}{81}Z+\frac{10k^2}{9}\bigg),
\end{align*}
and the conclusion follows easily.
\end{proof}

\FloatBarrier




\begin{thebibliography}{Flat}

\bibitem[BGK05]{BoyerGalickiKollarEinsteinMetricsOnSpheres}
Charles~P. Boyer, Krzysztof Galicki, and J\'{a}nos Koll\'{a}r, \emph{Einstein metrics on spheres}, Ann. of Math. (2) \textbf{162} (2005), no.~1, 557--580.

\bibitem[BH26]{BHNumericalEinsteinMetricOnTwelveSphere}
Timothy Buttsworth and Liam Hodgkinson, \emph{Computationally assisted proof of a novel {$ O(3)\times O(10)$}-invariant {E}instein metric on {$S^{12}$}}, J. Lond. Math. Soc. (2) \textbf{113} (2026), no.~2, Paper No. e70477, 44.

\bibitem[B{\"o}h98]{BohmInhomEinstein}
Christoph B{\"o}hm, \emph{{Inhomogeneous Einstein metrics on low-dimensional spheres and other low-dimensional spaces}}, Invent. Math. \textbf{134} (1998), no.~1, 145--176.

\bibitem[B{\"{o}}h99]{BohmNonCompactComhomOneEinstein}
Christoph B{\"{o}}hm, \emph{Non-compact cohomogeneity one {E}instein manifolds}, Bull. Soc. Math. France \textbf{127} (1999), no.~1, 135--177.

\bibitem[Chi24]{ChiPositiveEinsteinMetrics}
Hanci Chi, \emph{Positive {E}instein metrics with {$\Bbb {S}^{4m+3}$} as the principal orbit}, Compos. Math. \textbf{160} (2024), no.~5, 1004--1040.

\bibitem[CL55]{CoddingtonLevinsonODEs}
Earl~A. Coddington and Norman Levinson, \emph{Theory of ordinary differential equations}, McGraw-Hill Book Co., Inc., New York-Toronto-London, 1955.

\bibitem[EW00]{EschenburgWangInitialValueProblem}
J.-H. Eschenburg and McKenzie~Y. Wang, \emph{The initial value problem for cohomogeneity one {E}instein metrics}, J. Geom. Anal. \textbf{10} (2000), no.~1, 109--137.

\bibitem[FH17]{FoscoloHaskinsNearlyKaehler}
Lorenzo Foscolo and Mark Haskins, \emph{New {$G_2$}-holonomy cones and exotic nearly {K}\"{a}hler structures on {$S^6$} and {$S^3\times S^3$}}, Ann. of Math. (2) \textbf{185} (2017), no.~1, 59--130.

\bibitem[GK07]{GhigiKollarKaehlerEinsteinOrbifolds}
Alessandro Ghigi and J\'{a}nos Koll\'{a}r, \emph{K\"{a}hler-{E}instein metrics on orbifolds and {E}instein metrics on spheres}, Comment. Math. Helv. \textbf{82} (2007), no.~4, 877--902.

\bibitem[Jen73]{JensenEinsteinMetricsOnPrincipalFibreBundles}
Gary~R. Jensen, \emph{Einstein metrics on principal fibre bundles}, J. Differential Geometry \textbf{8} (1973), 599--614.

\bibitem[LST25]{LiuSanoTasinIninitelySasakiEinsteinMetricsSpheres}
Yuchen Liu, Taro Sano, and Luca Tasin, \emph{Infinitely many families of {S}asaki-{E}instein metrics on spheres}, J. Differential Geom. \textbf{130} (2025), no.~1, 1--26.

\bibitem[NW25]{EinsteinMetricsOnTheTenSpheres}
Jan Nienhaus and Matthias Wink, \emph{Einstein metrics on the Ten-Sphere}, to appear in J. Eur. Math. Soc., arXiv:2303.04832 (2025).

\bibitem[Wan26]{QiuShiWangComputerAssistedEinsteinMetrics}
Qiu~Shi Wang, \emph{Doubly warped product Einstein metrics on spheres}, arXiv:2606.05519 (2026).
  
\end{thebibliography}

\end{document}